\documentclass[12pt,a4paper,reqno]{amsart} 
\usepackage{amssymb}
\usepackage{latexsym}
\usepackage{amsmath}
\usepackage{mathrsfs}
\usepackage{calc}                         

\newcommand{\scal}[2]{\langle #1,#2\rangle}

\newcommand{\rr}[1]{\mathbf R^{#1}}

\newcommand{\nm}[2]{\Vert #1\Vert _{#2}}

\newcommand{\sets}[2]{\{ {\,}#1{\,};{\,}#2{\,}\} }

\newcommand{\fy}{\varphi}

\newcommand{\cdo}{\, \cdot \, }

\newcommand{\supp}{\operatorname{supp}}

\newcommand{\eabs}[1]{\langle #1\rangle}

\newcommand{\vrum}{\vspace{0.1cm}}
\newcommand{\FL}{\mathscr F \! L}
\newcommand{\masfR}{\mathsf R}

\numberwithin{equation}{section}          

\newtheorem{thm}{Theorem}
\numberwithin{thm}{section}
\newtheorem*{tom}{\rubrik}
\newcommand{\rubrik}{}
\newtheorem{prop}[thm]{Proposition}
\newtheorem{cor}[thm]{Corollary}
\newtheorem{lemma}[thm]{Lemma}

\theoremstyle{definition}

\theoremstyle{remark}

\newtheorem{rem}[thm]{Remark}              

\title{\textbf {Young type inequalities for weighted spaces}}
\author{Joachim Toft}

\address{Department of Computer science, Mathematics and Physics,
Linn{\ae}us University, V{\"a}xj{\"o}, Sweden}

\email{joachim.toft@lnu.se}

\author{Karoline Johansson}

\address{Department of Computer science, Mathematics and Physics,
Linn{\ae}us University, V{\"a}xj{\"o}, Sweden}

\email{karoline.johansson@lnu.se}

\author{Stevan Pilipovi\' c}

\address{Department of Mathematics and Informatics,
University of Novi Sad, Novi Sad, Serbia}

\email{stevan.pilipovic@dmi.uns.ac.rs}

\author{Nenad Teofanov}

\address{Department of Mathematics and Informatics,
University of Novi Sad, Novi Sad, Serbia}

\email{nenad.teofanov@dmi.uns.ac.rs}



\begin{document}

\begin{abstract}
We establish sharp convolution and multiplication estimates in
weighted Lebesgue, Fourier Lebesgue and modulation spaces. Especially
we recover some results in \cite{Hrm-nonlin,PTT2}.
\end{abstract}

\maketitle

\par

\section{Introduction}\label{sec0}

\par

The aim of the paper is to establish H{\"o}lder-Young type
properties for convolution and multiplications on weighted Lebesgue,
Fourier Lebesgue and modulation spaces.

\par

A frequently used convolution property concerns Young's inequality, which in
terms of the \emph{Young functional}
\begin{equation}\label{Rqfunctional}
\masfR (p)\equiv 2-\frac 1{p_0}-\frac 1{p_1}-\frac 1{p_2},\qquad p=
(p_0,p_1,p_2)\in [1,\infty]^3,
\end{equation}
asserts that
\begin{equation}\label{UsualYoung}
L^{p_1}_{t_1}*L^{p_2}_{t_2} \subseteq L^{p_0'}_{-t_0}
\end{equation}
when $\masfR (p)=0$,
\begin{equation}\label{firstsjskcond}
t_0+t_1\ge 0,\quad t_0+t_2\ge 0,\quad \text{and}\quad t_1+t_2\ge 0.
\end{equation}
We note that the latter inequalities imply
\begin{equation}\label{Usualscond}
t_0+t_1+t_2\ge 0.
\end{equation}
Here $L^p_t$ is the weighted Lebesgue space with parameters $p$
and $t$, and consists of all measurable functions $f$ on $\rr d$ such
that $f\cdot \eabs \cdo ^t\in L^p$, where $\eabs x=(1+|x|^2)^{1/2}$.
Furthermore, if $p\in [1,\infty ]$, then $p'\in [1,\infty ]$ is the conjugate
exponent for $p$, i.{\,}e. $1/p+1/p'=1$.

\par

Especially we are interested to find conditions on $t_j$, $j=0,1,2$, such that
\eqref{UsualYoung} holds when $\masfR (p)$ in \eqref{Rqfunctional}
stays somewhere in the interval $[0,2^{-1}]$.

\par

A rough estimate is obtained by an appropriate application of H{\"o}lder's
inequality on Young's inequality, here above. More precisely, by rewriting
$f_j(x)\cdot \eabs x ^{t_j}$ into $(f_j(x)\cdot \eabs x ^{\sigma _j})
\cdot \eabs x ^{t_j-\sigma _j}$, it follows by applying H{\"o}lder's
inequality on the $L^{p_j}_{t_j}$-norms in Young's inequality that the following
result holds true.

\par

\begin{prop}\label{weakyoungextension}
Let $t_j \in \mathbf R$, $p_j \in [1,\infty] $, $j=0,1,2$ and let $\masfR (p)$
be given by \eqref{Rqfunctional}. Also assume that $0< \masfR (p) \le 1/2$,
\eqref{firstsjskcond} holds true with at least two inequalities strict, and that
\begin{equation}\tag*{(\ref{Usualscond})$'$}
t_0+t_1+t_2>d\cdot \masfR (p)
\end{equation}
holds. Then $L^{p_1}_{t_1}*L^{p_2}_{t_2} \subseteq L^{p_0'}_{-t_0}$.
\end{prop}

\par

We remark that Proposition \ref{weakyoungextension} holds true also after
removing the condition $\masfR (q) \le 1/2$.

\par

In contrast to Young's inequality here above, the most of the inequalities in
\eqref{firstsjskcond} and \eqref{Usualscond}$'$ in Proposition
\ref{weakyoungextension} are strict, and if it is possible to replace
any such strict inequality by a non-strict one, then the situation is improved.
On the other hand, it seems not to be possible to perform such improvement
by only using H{\"o}lder's inequality in such simple way as described here
above.

\par

In this paper we use the framework in Chapter 8 in \cite{Hrm-nonlin}
and Section 3 in \cite{PTT2} to decompose the involved functions in the
convolutions in convenient ways. These investigations lead to
Theorem \ref{proptogether} in Section \ref{sec2}, which in particular gives
the following improvement of Proposition \ref{weakyoungextension}.

\par

\renewcommand{\rubrik}{Proposition \ref{weakyoungextension}$'$}

\begin{tom}
Let $t_j \in \mathbf R$, $p_j \in [1,\infty] $, $j=0,1,2$ and let $\masfR (p)$
be given by \eqref{Rqfunctional}. Also assume that $0< \masfR (p) \le 1/2$,
\eqref{firstsjskcond} holds true, and that
\begin{equation}\tag*{(\ref{Usualscond})$''$}
t_0+t_1+t_2\ge d\cdot \masfR(p)
\end{equation}
holds, with strict inequality in \eqref{Usualscond}$''$
when $t_j=d\cdot  \masfR (p) $ for some $j=0,1,2$. Then $L^{p_1}_{t_1}*
L^{p_2}_{t_2} \subseteq L^{p_0'}_{-t_0}$.
\end{tom}

\par

Furthermore, if $t_j\neq d\cdot \masfR (p)$, then we prove that
Proposition \ref{weakyoungextension}$'$ is optimal in the sense that if
\eqref{firstsjskcond} or \eqref{Usualscond}$''$ are violated, then
$L^{p_1}_{t_1}*L^{p_2}_{t_2}$ is \emph{not} continuously embedded in
$L^{p_0'}_{-t_0}$.

\par

Obviously, except for a few cases, the strict inequalities in Proposition
\ref{weakyoungextension} have been replaced by non-strict ones in
Proposition \ref{weakyoungextension}$'$. The H{\"o}rmander
theorem \cite[Theorem 8.3.1]{Hrm-nonlin} on microlocal regularity of a
product is obtained by choosing $p_0=p_1=p_2=2$
in Proposition \ref{weakyoungextension}$'$, and note that in
contrast to Proposition \ref{weakyoungextension}$'$, the latter
theorem is \emph{not} covered by Proposition \ref{weakyoungextension}.
We remark that the results in
\cite{Hrm-nonlin} are given in the framework of  weighted Sobolev
spaces of the form $H^2_s$, and the analysis is based on an intensive
use of their Hilbert space structure.
On the other hand, here (as well as in  \cite{PTT2}) our result considerations
include Banach spaces which might not be Hilbert spaces, and
thereby use a more sophisticated techniques in the proofs are needed.

\par

Finally we remark that Theorem \ref{proptogether} leads to Theorem
\ref{proptogetherMod} in Section \ref{sec2}, which concerns convolution
properties for modulation spaces. In particular, if $t_j$, $p_j$ and $\masfR (p)$
are the same as in Proposition \ref{weakyoungextension}$'$, and that
$$
\frac 1{q_0}+\frac 1{q_1}+\frac 1{q_2} =1,\quad \text{and}\quad
0\le s_0+s_1+s_2,
$$
then it follows from Theorem \ref{proptogetherMod} $M^{p_1,q_1}_{s_1,t_1}*
M^{p_2,q_2}_{s_2,t_2}\subseteq M^{p_0',q_0'}_{s_0,t_0}$.

\par

\section*{Acknowledgment}

\par

This research is supported by Ministry of Education, Science and
Technological Development of Serbia through the Project no.
174024.

\par

%
\section{Preliminaries}\label{sec1}

\par

In this section we review notions and notation, and
discuss basic preliminary results. We put $\mathbf N =\{0,1,2,\dots \}$,
and $A\lesssim B$ to indicate $A\leq c B$ for a suitable constant $c>0$.
Any extension of the $L^2$-scalar product on $C^\infty _0(\rr d)$
is denoted by $(\cdo  , \cdo )_{L^2} = (\cdo  , \cdo )$.

\par

The scalar product of $x$ and $\xi$  in $ \rr d $ is denoted by $\langle x,\xi \rangle$.
For $p \in [1,\infty]$ we let $p'\in [1,\infty ]$ denote the conjugate
exponent ($1/p+1/p'=1$).

\par

The Fourier transform $\mathscr F$ is the operator on $\mathscr S'(\rr d)$
which takes the form
$$
(\mathscr Ff)(\xi )= \widehat f(\xi ) \equiv (2\pi )^{-d/2}\int
f(x)e^{-i\scal  x\xi }\, dx, \qquad \xi \in \rr d,
$$
when $f\in L^1(\rr d)$.

\par

The (weighted) Fourier Lebesgue space $\mathscr FL^q_s(\rr d)$,
$s\in \mathbf R$ is the Banach space which consists of all $f\in \mathscr
S'(\rr d)$ such that
\begin{equation}\label{FLnorm}
\nm f{\mathscr FL^{q}_s} \equiv \nm {\widehat f\cdot \eabs \cdo ^s
}{L^q}
\end{equation}
is finite. Here and in what follows, $\eabs \xi =(1+|\xi |^2)^{1/2}$.

\par

Let $X $ be an open set in $ \rr d$. Then the \emph{local} Fourier
Lebesgue space $\mathscr FL^q_{s,loc}(X)$ consists of all
$f\in \mathscr D'(X)$ such that $\fy f\in \mathscr FL^q_{s}(\rr d)$ for
every $\fy \in C_0^\infty (X)$. The topology in
$\mathscr FL^q_{s,loc}(X)$  is defined by the family of
seminorms $f\mapsto  \nm {\fy f}{\mathscr FL^q_{(\omega )}}$,
where $\fy \in C_0^\infty (X)$.

\par

We note that
\begin{equation}\label{FLqFLqlocemb}
\mathscr FL^q_s(\rr d) \Big |_X \subseteq \mathscr FL^q_{s,loc}(X).
\end{equation}
and
\begin{equation}\label{incrFLloc}
\mathscr FL^{q_1}_{s_1,loc}(X)\subseteq \mathscr
FL^{q_2}_{s_2,loc}(X), \;\; \text{when} \ q_1\le q_2 \
\text{and} \ s_2\le s_1.
\end{equation}
(See e.{\,}g. \cite{PTT2}.)

\medspace

Next we define modulation spaces. Let $\phi \in \mathscr S'(\rr d)\setminus 0$
be fixed. Then the short-time Fourier transform of $f\in \mathscr S'(\rr d)$ with
respect to $\phi$ is defined by
$$
(V_\phi f)(x,\xi ) =\mathscr F(f\cdot \overline {\phi (\cdo -x)})(\xi ).
$$
Here the left-hand side makes sense, since it is the partial
Fourier transform of the tempered distribution
$ F(x,y)=(f\otimes \overline \phi)(y,y-x) $
with respect to the $y$-variable. We also note that if $f,\phi \in
\mathscr S(\rr d)$, then $V_\phi f$ takes the form
\begin{equation}\label{stftformula}
V_\phi f(x,\xi ) = (2\pi )^{-d/2}\int f(y)\overline {\phi
(y-x)}e^{-i\scal y\xi}\, dy.
\end{equation}

\par

Let $s,t\in \mathbf R$ and $p,q\in [1,\infty]$  be fixed. Then the modulation
space $M^{p,q}_{s,t}(\rr d)$ consists of all $f\in \mathscr S'(\rr d)$
such that
$$
\nm f{M^{p,q}_{s,t}}\equiv \left ( \int _{\rr d} \left ( \int _{\rr d} |V_\phi f(x,\xi )\eabs x^t
\eabs \xi ^s|^p\, dx  \right )^{q/p}d\xi  \right )^{1/q}
$$
is finite (with obvious interpretation of the integrals when $p=\infty$ or $q=\infty$). In the same
way, the modulation space $W^{p,q}_{s,t}(\rr d)$ consists of all $f\in \mathscr S'(\rr d)$
such that
$$
\nm f{W^{p,q}_{s,t}}\equiv \left ( \int _{\rr d} \left ( \int _{\rr d} |V_\phi f(x,\xi )\eabs x^t
\eabs \xi ^s|^q\, d\xi  \right )^{p/q}dx  \right )^{1/p}
$$
is finite.

\par

%
%

\section{Multiplication and convolution
properties}\label{sec2}

\par

In this section we derive multiplication and convolution results on Lebesgue,
Fourier Lebesgue and modulation spaces. In particular, we extend some
results in \cite{PTT2}. The proofs of the theorems
are postponed to Section \ref{sec4}.
Our main results are Theorems \ref{proptogether} and \ref{proptogetherMod}.
Here we present sufficient conditions on
$t_j\in \mathbf R$ and $p_j\in [1,\infty]$, $j=0,1,2$, to ensure that $v_1*v_2\in
L^{p_0'} _{-t_0} $ when $v_j \in L^{p_j} _{t_j}$,
$ j = 1,2$, and similarly when the convolution product and Lebesgue spaces are
replaced by multiplication and Fourier-Lebesgue spaces. The results also include
related multiplication and convolution properties for modulation spaces.

\par

%
%

Certain parts of the analysis concerns reformulation of $0\le \masfR
(p)\le 1/2$ into equivalent statements. For convenience we let 
\begin{equation}\label{Gdef}
G(x)=G(x_0,x_1,x_2)= 2-\sum _{j=0}^2 x _j,
\end{equation}
and note that $\masfR (p)$ is equal to $G(x)$ when $x_j=1/p_j$.
We also let
\begin{align}
H_0(x) &= \max _{\pi \in S_3} \left (
\min \left ( x_{\pi (0)},\max \left ( \frac 12,
\min (x_{\pi (1)},x_{\pi (2)})\right ) \right ) \right ).\label{H0def}
\\[1ex]
H_1(x) &=
\left \{
\begin{matrix}
\max (x_0,x_1,x_2),&\quad &x_0,x_1,x_2< \frac 12,
\\[1ex]
\min (x_0,x_1,x_2),&\quad &x_0,x_1,x_2 > \frac12,
\\[1ex]
\frac 12,\phantom{\min x_0,x_1,x_2}&\quad &\text{otherwise},
\end{matrix}
\right .\label{H1def}
\intertext{and}
H_2(x) &= \max \left ( \frac 12,\min (x_0,x_1,x_2)\right ),\label{H2def}
\end{align}
Here $S_3$ is the permutations of $\{ 0,1,2\}$. The following lemma justifies
the introduction of the functions $H_j(x)$, $j=0,1,2$, in
\eqref{H0def}--\eqref{H2def}.

\par

\begin{lemma}\label{1/2}
Let $x=(x_0,x_1,x_2)$, and let $G(x)$ and $H_l(x)$,
$l=0,1,2$, be given by \eqref{Gdef}--\eqref{H2def}. Then $H_0(x)=H_1(x)$.
Furthermore, if $l\in \{ 0,1,2\}$, then the following conditions are
equivalent.
\begin{enumerate}
\item $\displaystyle{0\le G(x)\le \frac 12}$;

\vrum

\item $\displaystyle{0\le G(x)\le H_l(x)}$.
\end{enumerate}
\end{lemma}

\par

\begin{proof}
We begin to prove $H_0(x)=H_1(x)$. We have $H_0(x)=\max
(y_0,y_1,y_2)$, where
\begin{align*}
y_0 &= \min \left ( x_0,\max \left ( \frac 12, \min (x_1,x_2) \right ) \right ),
\\[1ex]
y_1 &= \min \left ( x_1,\max \left ( \frac 12, \min (x_0,x_2) \right ) \right ),
\\[1ex]
y_2 &= \min \left ( x_2,\max \left ( \frac 12, \min (x_0,x_1) \right ) \right ).
\end{align*}
If $x_j\le 1/2$, then $y_j=x_j$, $j=0,1,2$, giving that
$$
H_0(x) =\max (x_0,x_1,x_2) =H_1(x)
$$
in this case.

\par

If instead $x_j\ge 1/2$, then $y_j=\min (x_0,x_1,x_2)$, $j=0,1,2$, giving that
$$
H_0(x) =\min (x_0,x_1,x_2) =H_1(x),
$$
in this case as well.

\par

Next assume that $x_j>1/2$ and $x_k<1/2$, for some choices of
$j,k\in \{ 0,1,2\}$. By reasons of symmetry we may assume that
$x_0=\min (x_0,x_1,x_2)<1/2$ and $x_1 =\max (x_0,x_1,x_2)>1/2$. Then
$$
H_1(x)=\frac 12,\quad y_0=x_0,\quad y_1=\frac 12\quad \text{and}\quad
y_2=\min \left ( x_2,\frac 12\right ) \le \frac 12.
$$
Hence,
$$
H_0(x) =\max (y_0,y_1,y_2) =\frac 12 = H_1(x).
$$
which shows that $H_0(x)=H_1(x)$ for all $x$.

\par

It remains to prove the equivalence between (1) and (2). It is obvious that (2) with
$l=1$ or (1) implies (2) with $l=2$. Next assume that (2) with $l=2$ holds but not (1).
Then $G(x)>1/2$ and $H_2(x)>1/2$, which implies that
$\min \{ x_0,x_1,x_2 \} >1/2$. This gives
$$
G(x)= 2-\sum _{j=0}^2 x _j <2-\frac{3}{2} = \frac{1}{2},
$$
which is a contradiction. Hence (2) with $l=2$ implies (1), and we have proved the
equivalence between (1) and (2) when $l=2$.

\par

Since $H_0(x)=H_1(x)=H_2(x)$ when $x_j\ge 1/2$ for some $j=0,1,2$, it
suffices to consider the case
$x_j<1/2$, $j=0,1,2$, when proving that (2) is invariant under the choice of $l=0,1,2$.
Then by the first part of the proof we have $H_0(x)=H_1(x)<1/2$, $H_2(x)=1/2$ and
$$
G(x)=2-\sum _{j=0}^2 x_j >1/2.
$$
Hence (2) is violated in this case for any $j\in \{ 0,1,2\}$. This proves the
invariance of (2) under the choice of $j$, and the proof is complete.
%
%
%
\end{proof}

\par

In the main results here below we consider convolutions between elements
in weighted Lebesgue and modulation spaces, and multiplications between
elements in (weighted) Fourier-Lebesgue spaces.
For the convolution results, the parameters on the weights
should satisfy
\begin{align}
0&\leq t_j+t_k,  \qquad j,k=0,1,2,  \quad j\neq k,\label{lastineq2A}
\\[1ex]
0 &\leq t_0+t_1 + t_2 - d \cdot  \masfR (p),
\label{lastineq2B}
\intertext{and}
0 &\le s_0+s_1+s_2.\label{lastineq2C}
\end{align}
(Cf. \eqref{firstsjskcond} and \eqref{Usualscond}$'$.)
If the convolution is replaced by multiplication, then the roles for $p_j$ and
$q_j$, and for $s_j$ and $t_j$ are interchanged. Therefore,
\eqref{lastineq2A}--\eqref{lastineq2C}
should be replaced by
\begin{align}
0&\leq s_j+s_k,  \qquad j,k=0,1,2,  \quad j\neq k,\tag*{(\ref{lastineq2A})$'$}
\\[1ex]
0 &\leq s_0+s_1 + s_2 - d \cdot  \masfR (q),
\tag*{(\ref{lastineq2B})$'$}
\intertext{and}
0 &\le t_0+t_1+t_2,\tag*{(\ref{lastineq2C})$'$}
\end{align}
when the Lebesgue parameters are $q$ and $q_j$
instead of $p$ and $p_j$, respectively.
%

\par

\begin{thm}\label{proptogether}
Let $s_j,t_j \in \mathbf R$, $p_j,q_j \in
[1,\infty] $, $j=0,1,2$ and let $\masfR $ be the functional in
\eqref{Rqfunctional}. Then the following is true:
\begin{enumerate}
\item Assume that $0\le \masfR (p) \le 1/2$, and that
\eqref{lastineq2A} and \eqref{lastineq2B}
hold true with strict inequality in \eqref{lastineq2B}
when $\masfR (p)>0$ and $s_j=d\cdot  \masfR (p) $
for some $j=0,1,2$. Then the map $(f_1,f_2)\mapsto f_1*f_2$
on $C_0^\infty (\rr d)$ extends uniquely to a continuous map from
$ L^{q_1} _{t_1}(\rr d) \times  L^{q_2} _{t_2}(\rr d)$ to $L^{q_0'} _{-t_0}(\rr d)$;

\vrum

\item Assume that $0\le \masfR (q) \le 1/2$, and that \eqref{lastineq2A}$'$ and
\eqref{lastineq2B}$'$ hold true with strict inequality in \eqref{lastineq2B}$'$
when $\masfR (q)>0$ and $s_j=d\cdot  \masfR (q) $
for some $j=0,1,2$. Then the map $(f_1,f_2)\mapsto f_1\cdot f_2$ on $C_0
^\infty (\rr d)$ extends uniquely to a continuous map from $ \mathscr F
L^{q_1} _{s_1}(\rr d)\times  \mathscr F L^{q_2} _{s_2}(\rr d)$ to $\mathscr F
L^{q_0'} _{-s_0}(\rr d)$.
\end{enumerate}
\end{thm}

\par

The following corollary follows immediately from \eqref{FLqFLqlocemb} and
Theorem \ref{proptogether}.

\par

\begin{cor} \label{coro-local}
Let the hypothesis in Theorem \ref{proptogether} hold true, and
let $X\subseteq \rr r$ be open. Then the map $(f_1,f_2)\mapsto f_1
\cdot f_2$
on $C_0^\infty (X)$ extends uniquely to a continuous map from
$ (\mathscr F L^{q_1} _{s_1})_{loc}(X) \times  (\mathscr F L^{q_2}
_{s_2})_{loc}(X)$ to $(\mathscr F L^{q_0'} _{-s_0})_{loc}(X)$
\end{cor}

\par

The next result concerns corresponding properties for modulation spaces.

\par

\begin{thm}\label{proptogetherMod}
Let the hypothesis in Theorem \ref{proptogether} hold true and
Then the following is true:
\begin{enumerate}
\item Assume that $0\le \masfR (p) \le 1/2$, $\masfR (q)\le 1$ and
\eqref{lastineq2A}--\eqref{lastineq2C}
hold true, with strict inequality in \eqref{lastineq2B}
when $\masfR (p)>0$ and $t_j=d\cdot  \masfR (p) $
for some $j=0,1,2$. Then the map $(f_1,f_2)\mapsto f_1*f_2$
on $C_0^\infty (\rr d)$ extends to a continuous map from
$M^{p_1,q_1} _{s_1,t_1}(\rr d) \times  M^{p_2,q_2} _{s_2,t_2}(\rr d)$ to
$M^{p_0',q_0'} _{-s_0,-t_0}(\rr d)$;
%
%

\vrum

\item Assume that $\masfR (p)\le 1$ and $0\le \masfR (q) \le 1/2$, and
\eqref{lastineq2A}$'$--\eqref{lastineq2C}$'$
hold true, with strict inequality in \eqref{lastineq2B}$'$
when $\masfR (q)>0$ and $s_j=d\cdot  \masfR (q) $
for some $j=0,1,2$. Then the map $(f_1,f_2)\mapsto f_1\cdot f_2$ on
$C_0^\infty (\rr d)$ extends to a continuous map from $M^{p_1,q_1}
_{s_1,t_1}(\rr d) \times M^{p_2,q_2} _{s_2,t_2}(\rr d)$ to $M^{p_0',q_0'}
_{-s_0,-t_0}(\rr d)$.

\par

\end{enumerate}

\par

The same is true after $M^{p_j,q_j}_{s_j,t_j}$
have been replaced by $W^{p_j,q_j}_{s_j,t_j}$, $j=0,1,2$.

\par

Furthermore, the extensions of these mappings are unique, except when
$p_j$ or $q_j$ are equal to $\infty$ for more than one choice of $j=0,1,2$.
%
%
\end{thm}

\par

\begin{rem}
By letting $x_j=1/p_j$ in Lemma \ref{1/2}, we may replace the condition
$0\le \masfR (p)\le 1/2$ in Theorems \ref{proptogether}--\ref{proptogetherMod}
with
$$
0 \le \masfR (p) \le \max \left ( \frac{1}{2}, \min \left (
\frac {1}{p_0}, \frac {1}{p_1},\frac {1}{p_2} \right ) \right ).
$$
\end{rem}

\par

The following result shows that the conditions \eqref{lastineq2A} and
\eqref{lastineq2B} are also necessary in order for the continuity in Theorem
\ref{proptogether} should hold true.

\par

\begin{prop}\label{otpmimality1}
Let $p_j,q_j\in [1,\infty ]$ and $s_j,t_j\in \mathbf R$, $j=0,1,2$. Assume that at
least one of the following statements hold true:
\begin{enumerate}
\item the map $(f_1,f_2)\mapsto f_1*f_2$ on $\mathscr S(\rr d)$ is continuously
extendable to a map from $L^{p_1}_{t_1}(\rr d)\times L^{p_2}_{t_2}(\rr d)$ to
$L^{p_0'}_{-t_0}(\rr d)$;

\vrum

\item the map $(f_1,f_2)\mapsto f_1*f_2$ on $\mathscr S(\rr d)$ is continuously
extendable to a map from $M^{p_1,q_1}_{s_1,t_1}(\rr d)\times
M^{p_2,q_2}_{s_2,t_2}(\rr d)$ to $M^{p_0',q_0'} _{-s_0,-t_0}(\rr d)$;

\vrum

\item the map $(f_1,f_2)\mapsto f_1*f_2$ from $\mathscr S(\rr d) \times \mathscr
S(\rr d)$ to $\mathscr S(\rr d)$ is continuously extendable to a map from
$W^{p_1,q_1}_{s_1,t_1}(\rr d)\times W^{p_2,q_2}_{s_2,t_2}(\rr d)$ to $W^{p_0',q_0'}
_{-s_0,-t_0}(\rr d)$.
\end{enumerate}
Then \eqref{lastineq2A} and \eqref{lastineq2B} hold true.
\end{prop}

\par

By Fourier transformation, it follows that Proposition \ref{otpmimality1}
is equivalent to the following result.

\par

\begin{prop}\label{otpmimality2}
Let $p_j,q_j\in [1,\infty ]$ and $s_j,t_j\in \mathbf R$, $j=0,1,2$. Assume that at
least one of the following statements hold true:
\begin{enumerate}
\item the map $(f_1,f_2)\mapsto f_1\cdot f_2$ on $\mathscr S(\rr d)$ is continuously
extendable to a map from $\mathscr FL^{q_1}_{s_1}(\rr d)\times
\mathscr FL^{q_2}_{s_2}(\rr d)$ to $\mathscr FL^{q_0'}_{-s_0}(\rr d)$;

\vrum

\item the map $(f_1,f_2)\mapsto f_1\cdot f_2$ on $\mathscr S(\rr d)$ is continuously
extendable to a map from $M^{p_1,q_1}_{s_1,t_1}(\rr d)\times
M^{p_2,q_2}_{s_2,t_2}(\rr d)$ to $M^{p_0',q_0'} _{-s_0,-t_0}(\rr d)$;

\vrum

\item the map $(f_1,f_2)\mapsto f_1\cdot f_2$ from $\mathscr S(\rr d) \times \mathscr
S(\rr d)$ to $\mathscr S(\rr d)$ is continuously extendable to a map from
$W^{p_1,q_1}_{s_1,t_1}(\rr d)\times W^{p_2,q_2}_{s_2,t_2}(\rr d)$ to $W^{p_0',q_0'}
_{-s_0,-t_0}(\rr d)$.
\end{enumerate}
Then \eqref{lastineq2A}$'$ and \eqref{lastineq2B}$'$ hold true.
\end{prop}

\par

\begin{rem}
In the literature, there are several results which are related to Theorems
\ref{proptogether} and \ref{proptogetherMod} (cf. e.{\,}g. the first part of
Proposition 2.3 in \cite{LiuCui} and the references therein). It seems that some of
these results contain some mistakes.

\par

More precisely, let $p_j,q_j\in [1,\infty]$, $j=0,1,2$, and $s\ge 0$ be such
that $\masfR (p)=1$ and $\masfR (q) =0$. Then it is remarked in \cite{Iw}
that the map $(f_1,f_2)\mapsto f_1\cdot f_2$ on $\mathscr S$ is
extendable to a continuous map from $M^{p_1,q_1}_{s,0}\times
M^{p_2,q_2}_{0,0}$ to $M^{p_0',q_0'}_{s,0}$. (Cf. Remark 2.4 in \cite{Iw}.)
We claim that this is not correct when $s>0$.

\par

In fact, by applying the Fourier transform and using duality, the statement is
equivalent to the following statement:

\medspace

\emph{Let $p_j,q_j\in [1,\infty ]$ and $t_j\in \mathbf R$, $j=0,1,2$, be such that
$\masfR (p)=0$, $\masfR (q) =1$, $t_1=-t_2\ge 0$ and $t_0=0$. Then the
map $(f_1,f_2)\mapsto f_1\cdot f_2$ on $\mathscr S$ is extendable to
a continuous map from $W^{p_1,q_1}_{0,t_1}\times W^{p_2,q_2}_{0,t_2}$ to
$W^{p_0',q_0'}_{0,t_0}$.}

\medspace

The hypothesis in Proposition \ref{otpmimality1} is therefore fulfilled, but
\eqref{lastineq2A} is violated. This contradicts Proposition \ref{otpmimality1},
and the claim follows.
\end{rem}

\par

\section{Proofs} \label{sec4}

\par

In this section we present proofs of the results in Section
\ref{sec2}. In Subsection \ref{subsec4.1} we study in details the
problem of extensions of an auxiliary three-linear map. In
Subsection \ref{subsec4.2} we use the results from Subsection
\ref{subsec4.1} to prove Lebesgue norm estimates of the three-linear
form on different regions. Finally, in Subsection \ref{subsec4.3} we
prove the main results.

\par

\subsection{The map $T_F(f,g)$}\label{subsec4.1}
In this subsection we introduce and study a
convenient bilinear map (denoted by
$T_F $ here below when $F\in L^1_{loc}$ is appropriate).
We refer to \cite{Hrm-nonlin} and \cite{PTT2} for similar construction.

\par

For $F \in L^1 _{loc} (\rr {2d}) $ and $ p,q \in [1,\infty] $, we
set
\begin{align*}
\nm F{L^{p,q}_1} &\equiv \Big ( \int \Big ( \int |F(x ,y )|^p\,
dx \Big )^{q/p}\, dy  \Big )^{1/q}
\intertext{and} \nm
F{L^{p,q}_2} &\equiv \Big ( \int \Big ( \int |F(x ,y )|^q\,
dy \Big )^{p/q}\, dx  \Big )^{1/p},
\end{align*}
and we let $L^{p,q}_1(\rr {2d})$ be the set of all $F \in L^1 _{loc} (\rr {2d})$ such that
$\nm F{L^{p,q}_1}$ is finite. The space $L^{p,q}_2$ is defined analogously.  (Cf.
\cite{PTT, PTT2}.) We also let $\Theta$ be defined as
\begin{equation}\label{FGrel}
(\Theta F)(x ,y ) = F(x ,x-y ),\qquad F\in L^1_{loc}(\rr {2d}).
\end{equation}

\par

If $F\in L^1_{loc}(\rr {2d})$ is fixed, then we are especially concerned about
extensions of the mappings
\begin{alignat}{2}
(F,f,g) &\mapsto & T_F(f,g) &\equiv \int F(\cdo  ,y )f(y )g(\cdo  -y )\, dy
\label{TFmap}
\intertext{and}
(F,f,g) &\mapsto & T_{\Theta F}(f,g) &\equiv \int F(\cdo  ,y )f(\cdo  - y )g(y )\, dy .
\label{TGmap}
\end{alignat}
from $ C^\infty _0 (\rr d) \times  C^\infty _0 (\rr d)$
to $  \mathscr S' (\rr d)$.

\par

The following extend \cite[Lemma 8.3.2]{Hrm-nonlin} and \cite[Proposition
3.2]{PTT2}.

\par

\begin{prop}\label{prop1}
Let $F\in L^1_{loc}(\rr {2d})$, $p_j\in [1,\infty
]$, $j=0,1,2$. Also assume that $\masfR (p)$ in \eqref{Rqfunctional} is non-negative,
and let $r=1/\masfR (p)\in (0,\infty ]$.
Then the following is true:
\begin{enumerate}
\item if $\masfR (p) \leq 1/p_0$, then the mappings \eqref{TFmap}
and \eqref{TGmap} are continuous from $L^{\infty ,r}_2(\rr
{2d})\times L^{p_1}(\rr d)\times L^{p_2}(\rr d)$ to $L^{p'_0}(\rr
d)$. Furthermore,
\begin{align}
\nm{T_F(f,g)}{L^{q'_0}} &\lesssim
\nm{F}{L^{\infty,r}_2}\nm{f}{L^{p_1}}\nm{g}{L^{p_2}}\label{TFL2est1}
\intertext{and} \nm{T_{\Theta F}(f,g)}{L^{p_0'}} &\lesssim
\nm{F}{L^{\infty,r}_2}\nm{f}{L^{p_1}}\nm{g}{L^{p_2}}.
\label{TFL2est2}
\end{align}

\vrum

\item if in addition $\masfR (p)\le \max (1/2,1/p_1)$,
then the map \eqref{TFmap} is continuous from $L^{r,\infty}_1(\rr
{2d})\times L^{p_1}(\rr d)\times L^{p_2}(\rr d)$ to $L^{p'_0}(\rr
d)$. Furthermore,
\[
\nm{T_F(f,g)}{L^{p'_0}} \lesssim \nm{F}{L^{r,\infty}_1}
\nm{f}{L^{p_1}}\nm{g}{L^{p_2}}.
\]

\vrum

\item if in addition $\masfR (p)\le \max (1/2,1/p_2)$,
then the map \eqref{TGmap} is continuous from $L^{r,\infty}_1(\rr
{2d})\times L^{p_1}(\rr d)\times L^{p_2}(\rr d)$ to $L^{p'_0}(\rr d)$.
Furthermore, 
\[
\nm{T_{\Theta F}(f,g)}{L^{p'_0}} \lesssim \nm{F}{L^{r,\infty}_1}
\nm{f}{L^{p_1}}\nm{g}{L^{p_2}}.
\]
\end{enumerate}
\end{prop}

\par

We note that Proposition \ref{prop1} agrees with \cite[Lemma
8.3.2]{Hrm-nonlin} when $ p_1 = p_2 = 2 $ and with \cite[Proposition
3.2]{PTT2} when $ p_1 = p_2 \in [1,\infty] $.

\par

\begin{proof}
(1) We only prove \eqref{TFL2est1} and leave $\eqref{TFL2est2}$ for the reader.

\par

First, assume that $p_1,p_2<\infty$, and let $f,g\in C_0^\infty (\rr d)$. By
H{\"o}lder's inequality we get
\begin{multline}
\Big(\int |T_F(f,g)(x)|^{q'_0} \, d x\Big)^{1/q'_0}
\\[1 ex]
\leq \Big ( \int \Big [ \Big ( \int |F(x ,y )|^{r} \, dy
\Big )^{1/r} \Big ( \int | f(y)|^{r'} |g(x -y
)|^{r'}\, dy \Big) ^{1/r'} \Big ] ^{q'_0}\, dx \Big ) ^{1/q'_0}.
\end{multline}
Next we use the assumption $\masfR (q) \leq 1/p_0$, that is
$ r \geq p_0$  and Young's inequality to obtain
\begin{multline}\label{TFcomputations}
\Big(\int |T_F(f,g)(x)|^{q'_0} \, d x\Big)^{1/q'_0}
\\[1 ex]
\leq \nm {F}{L^{\infty ,r}_2} \Big ( \nm
{|f|^{r'}*|g|^{r'}}{L^{q'_0/r'}} \Big )^{1/r'} \leq \nm
{F}{L^{\infty,r}_2} \Big ( \nm {|f|^{r'}}{L^{r_1}} \nm
{|g|^{r'}}{L^{r_2}} \Big )^{1/r'}
\\[1 ex]
= \nm {F}{L^{\infty,r}_2} \nm{f}{L^{p_1}}\nm{g}{L^{p_2}},
\end{multline}
where $ r_1 = p_1/r' $ and $r_2 = p_2/r'$.
The result now follows from the fact that $C^\infty_0$ is dense in
$L^{p_1}$ and $L^{p_2}$ when $p_1, p_2 < \infty$.

\par

Next, assume that $p_1=\infty$ and $p_2<\infty$, and let $f\in
L^\infty$ and $g\in C_0^\infty$. Then, it follows that $T_F(f,g)$ is
well-defined, and that \eqref{TFcomputations} still holds. The
result now follows from the fact that $C_0^\infty$ is dense in
$L^{p_2}$. The case
$p_1<\infty$ and $p_2=\infty$ follows analogously.

\par

Finally, if $p_1=p_2=\infty$, then the assumptions
implies that $r=1$ and $p'_0=\infty$. The inequalities
\eqref{TFL2est1} and \eqref{TFL2est2} then follow  by H{\"o}lder's inequality.

\par

(2) First we consider the case $r\ge p_1$. Let $h\in C_0(\rr d)$
when $r <\infty$ and $h\in L^1(\rr d)$ if $r=\infty$. Also let
$F\in L^{r,\infty}_1(\rr {2d})$ and $F_0(y,x)=F(x,y)$
and $\check g(x)=g(-x)$. By \cite[page 354]{PTT2}, we have
$ |\left \langle T_F (f,g) , h \right \rangle | = | \left \langle
T_{F_0}(h,\check g),f \right \rangle | $.
Then (1) implies
\begin{multline*}
|\left \langle T_F (f,g) , h \right \rangle | = | \left \langle
T_{F_0}(h,\check g),f \right \rangle |
\\[1 ex]
\leq \nm {T_{F_0}(h, \check g)}{L^{p_1'}} \nm {f}{L^{p_1}} \leq
\nm{F_0}{L_2^{\infty ,r}} \nm {f}{L^{p_1}} \nm {h}{L^{p_0}}
\nm{g}{L^{p_2}}
\\[1 ex]
\leq
\nm{F}{L_1^{r,\infty}} \nm {f}{L^{p_1}} \nm {h}{L^{p_0}}
\nm{g}{L^{p_2}}.
\end{multline*}

\par

Next, assume that $r \ge 2$ and $F\in L^{r,\infty}_1(\rr {2d})$.
We will prove the assertion by interpolation.
First we consider the case $r=\infty$. Then
$ \masfR (p) = 0$, and
\begin{multline*}
\Big \Vert \int F(x,y)f(y)g(x-y)\, dy \Big \Vert
_{L^{p'_0}} \leq \nm{F}{L^{\infty,\infty}_1} \nm{|f|*|g|}{L^{p'_0}}
\\[1 ex]
\leq
\nm{F}{L^{\infty,\infty}_1}\nm{f}{L^{p_1}}\nm{g}{L^{p_2}}.
\end{multline*}

\par

For the case $r=2$ we have $\masfR (p)=1/2$.
By letting
$$
M =\nm {F}{L^{2,\infty}_1}, \quad  \theta = \frac {( \nm {g}{L^{2r_1}}
\nm{h}{L^{2r_2}} )^{1/p_1} }{ \nm {f}{L^{p_1}}^{1/p_1'}},\quad
r_1=p_2/2\quad \text{and} \quad r_2=p_0/2,
$$
it follows from Cauchy-Schwartz inequality, the weighted
arithmetic-geometric mean-value inequality and Young's inequality that
\begin{multline*}
|\left \langle T_F(f,g),h \right \rangle |
%
%
\leq \int \Big ( \int |F(x ,y )| |g(x -y )| |h(x )|\,
dx \Big ) |f(y )| \, dy
\\[1 ex]
\le M \int \Big ( \int |g(x - y) |^2 |h(x )|^2 \, dx \Big )
^{1/2} |f(y )| \, d y
\end{multline*}
\\[-4ex]
$$
\leq M \int \Big ( \frac{\theta^{p_1}}{p_1} |f(y)|^{p_1} +
\frac{1}{p_1' \theta^{p_1'}} \Big ( \int |g(x -y )|^2
|h(x )|^2 \, dx \Big ) ^{p_1'/2} \Big ) \, dy
$$
\begin{multline*}
=  M \Big ( \frac{\theta^{p_1}}{p_1} \nm{f}{L^{p_1}}^{p_1} +
\frac{1}{p_1' \theta^{p_1'}} \nm {|g|^2 * |h| ^2}
{L^{p_1'/2}}^{p_1'/2} \Big )
\\[1 ex]
\leq M \Big ( \frac {\theta ^{p_1}}{p_1} \nm {f}{L^{p_1}}^{p_1} +
\frac{1}{p_1' \theta ^{p_1'}} ( \nm {|g|^2}{L^{r_1}} \nm
{|h|^2}{L^{r_2}})^{p_1'/2} \Big )
\\[1 ex]
= M \Big ( \frac { \theta^{p_1}}{p_1} \nm {f}{L^{p_1}}^{p_1} +
\frac {1}{p_1' \theta ^{p_1'}} (\nm {g} {L^{2r_1} } \nm
{h}{L^{2r_2}}) ^{p_1'} \Big)
\\[1 ex]
=  M \Big (\frac {1}{p_1}+\frac {1}{p_1'} \Big ) \nm {f}{L^{p_1}}\nm
{g}{L^{p_2}} \nm {h}{L^{p_0}}
\\[1 ex]
=  M \nm {f}{L^{p_1}} \nm {g}{L^{p_2}} \nm{h}{L^{p_0}}.
\end{multline*}
This gives the result for $r=2$.

\par

Since we also have proved the result for $r=\infty$. The
assertion (2) now follows for general $r\in [2,\infty ]$ by
multi-linear interpolation, using Theorems 4.4.1, 5.1.1 and 5.1.2 in \cite{BL}.
%

\par

The assertion (3) follows by similar arguments as in the proof of
(2). The details are left for the reader. The proof is complete.

\end{proof}

\subsection{Some lemmas}\label{subsec4.2}
Before the proof of Theorem \ref{proptogether}, we need some preparation,
and formulate auxiliary results in three Lemmas.

\par

First, we recall \cite[Lemma 3.5]{PTT2} which concerns
different integrals of the function
\begin{equation}\label{Fdef}
F(x ,y ) = \eabs x ^{-t_0}\eabs {x -y}^{-t_1}\eabs y
^{-t_2}, \quad x ,y \in \rr d,
\end{equation}
where $t_j\in \mathbf R$, $j=0,1,2$. These integrals, with respect to
$x$ or $y$, are taken
over the sets
\begin{equation}\label{theomegasets}
\begin{aligned}
\Omega _1 &= \sets{ (x, y) \in \rr {2d}}{\eabs y < \delta
\eabs x },
\\[1ex]
\Omega _2 &= \sets{ (x, y) \in \rr {2d}}{\eabs {x -y }<
\delta \eabs x },
\\[1ex]
\Omega _3 &= \sets{ (x, y) \in \rr {2d}}{\delta \eabs x \le
\min (\eabs y ,\eabs {x -y }),\ |x |\le R},
\\[1ex]
\Omega _4 &= \sets{ (x, y) \in \rr {2d}}{\delta \eabs x
  \le \eabs {x -y }\le  \eabs y,\ |x |> R},
\\[1ex]
\Omega _5 &= \sets{ (x, y) \in \rr {2d}}{\delta \eabs x \le
  \eabs y \le \eabs {x -y },\ |x |> R},
\end{aligned}
\end{equation}
for some positive constants $\delta$ and $R$.
By $ \chi _{\Omega _j} $ we denote the characteristic function
of the set $ \Omega_j,$ $ j = 1,\dots, 5.$

\par

\begin{lemma}\label{intestimates}
Let $F$ be given by \eqref{Fdef} and let $\Omega _1,\dots,\Omega _5$
be given by \eqref{theomegasets}, for some
constants $0<\delta <1$ and $R\ge 4/\delta$. Also let $p\in [1,\infty ]$
and $ F_j=\chi _{\Omega _j}F$, $ j = 1,\dots, 5.$ Then 
the following is true:
\begin{enumerate}
\item
$$
\nm {F_1(x ,\cdo  )}{L^p}\lesssim
\begin{cases}
\eabs x ^{-t_0-t_1}\big ( 1+\eabs x
^{-t_2+d/p} \big ),& t_2\neq d/p ,
\\[1ex]
\eabs x ^{-t_0-t_1}\big ( 1+\log
\eabs x  \big )^{1/p},& t_2= d/p \text;
\end{cases}
$$

\vrum

\item
$$
\nm {F_2(x ,\cdo  )}{L^p}\lesssim
\begin{cases}
\eabs x ^{-t_0-t_2}\big ( 1+\eabs x
^{-t_1+d/p} \big ),& t_1\neq d/p ,
\\[1ex]
\eabs x ^{-t_0-t_2}\big ( 1+\log
\eabs x  \big )^{1/p},& t_1= d/p \text;
\end{cases}
$$

\vrum

\item $\nm {F_3(\cdo ,y )}{L^p}\lesssim \eabs y ^{-t_1-t_2}$;

\vrum

\item if $j=4$ or $j=5$, then
$$
\nm {F_j(\cdo ,y )}{L^p}\lesssim
\begin{cases}
\eabs y ^{-t_0-t_1-t_2+d/p},& t_0 < d/p ,
\\[1ex]
\eabs y ^{-t_1-t_2}\big ( 1+\log
\eabs y  \big )^{1/p},& t_0 = d/p,
\\[1ex]
\eabs y ^{-t_1-t_2},& t_0 > d/p .
\end{cases}
$$
\end{enumerate}
\end{lemma}

\par

We refer to \cite{PTT2} for the proof of Lemma \ref{intestimates}.

\par

Next we estimate each of the auxiliary
functions $ T_{F_j}, $
defined by \eqref{TFmap} with $F$ replaced by $ F_j $, $ j = 1,\dots, 5$.

\par

\begin{lemma} \label{estimatesTFj}
Let $\masfR (p)$, $F$ and $T_F $ be given by \eqref{Rqfunctional}, \eqref{Fdef}
and \eqref{TFmap} respectively, and let
$\Omega _1,\dots,\Omega _5$ be given by \eqref{theomegasets}, for some
constants $0<\delta <1$ and $R\ge 4/\delta$.
Moreover, let
$F_j=\chi _{\Omega _j}F$, $ j = 1,\dots, 5$, and
$u_j = \eabs \cdo ^{t_j} f_j$, $ j=1,2$.
Then the estimate
\begin{equation*}
\nm {T_ {F_j} (u_1, u_2)}{L^{p_0'}}
\lesssim  \nm {f_1}{L ^ {p_1} _{t_1}}  \nm {f_2}{L ^ {p_2} _{t_2}}
\end{equation*}
holds when:
\begin{enumerate}
\item $j=1,2$, for $\masfR (p)\le 1/p_0$, $0 \leq t_0 + t_1$, $ 0\leq t_0 + t_2$ and
$$
0  \leq t_0 + t_1 + t_2 - d \cdot \masfR (p),
$$
where the above inequality is strict when $ t_1 = d\cdot \masfR (p)$
or $  t_2 = d \cdot \masfR (p)$.

\vrum

\item $j=3$, for
$$
\begin{cases}
\masfR (p) \le \min(1/p_1,1/p_2) & \text {when} \quad  p_1, p_2 < 2 ,
\\[1 ex]
\masfR (p) \le 1/2_{\phantom 2}
& \text {when}  \quad p_1\geq 2 \quad \text{or} \quad p_2\geq 2,
\end{cases}
$$
and
$$
0 \leq t_1 + t_2\text ;
$$

\vrum

\item $j=4$ for
$
\masfR (p) \le \max(1/p_2,1/2),
%
$
$$
0 \leq t_1 + t_2
\quad
\text{and}
\quad
0  \leq t_0 + t_1 + t_2 - d\cdot \masfR (p),
$$
with $ 0< t_1 + t_2 $ when
$ t_0 =   d \cdot \masfR (p)$;

\vrum

\item $j=5$, for
$
\masfR (p) \le \max(1/p_1,1/2),
$
$$
0 \leq t_1 + t_2
\quad
\text{and}
\quad
0  \leq t_0 + t_1 + t_2 - d\cdot \masfR (p),
$$
with $ 0< t_1 + t_2 $ when $ t_0 =  d \cdot \masfR (p)$.
\end{enumerate}
\end{lemma}

\par

\begin{proof}
Let $r=1/\masfR (p)$.

\par

(1) The condition $\masfR (p)\le 1/p_0$ implies that $ r \geq p_0'$.
%
By Lemma \ref{intestimates} (1) it follows that
\begin{equation}\label{Fjest1}
\nm {F_1}{L^{\infty ,r}_2}<\infty
\end{equation}
when $0\leq t_0 + t_1$ and
$$
\begin{cases}
0\leq t_0 + t_1 + t_2 - d/r, \quad &\text {for} \quad  t_2 \neq d/ r
\\[1 ex]
0< t_0 + t_1,  \quad &\text {for} \quad  t_2 = d/ r.
\end{cases}
$$
Similarly, by Lemma \ref{intestimates} (2) it follows that
\begin{equation}\label{Fjest2}
\nm {F_2}{L^{\infty ,r}_2}<\infty
\end{equation}
when $0\leq t_0 + t_2$ and
$$
\begin{cases}
0 \leq t_0 + t_1 + t_2 - d/r, \quad &\text {for} \quad  t_1 \neq d/ r
\\[1 ex]
0 < t_0 + t_2, \quad &\text {for} \quad  t_1 = d / r.
\end{cases}
$$
This, together with Proposition \ref{prop1} (1) gives
$$
\nm {T_ {F_j} (u_1, u_2)}{L^{p_0'}}
\lesssim
 \nm {f_1}{L ^ {p_1} _{t_1}}  \nm {f_2}{L ^ {p_2} _{t_2}}, \quad j = 1,2.
$$

\par

(2) By Lemma \ref{intestimates} (3) we have
\begin{equation}\label{Fjest3}
\nm {F_3}{L^{r_0,\infty }_1}<\infty ,
\end{equation}
when  $t_1 + t_2\geq 0$ and $ r_0 \in [1,\infty]$.
In particular, if $r_0=r=1/\masfR (p)$ and $r \geq \min (2, \max (p_1, p_2))$, then
it follows from Proposition \ref{prop1} (2) and (3) that
\begin{equation*}
\nm {T_ {F_3} (u_1, u_2)}{L^{p'_0}}
\leq
C \nm {f_1}{L ^ {p_1} _{t_1}}\nm {f_2}{L ^ {p_2} _{t_2}}.
\end{equation*}
This gives (2).

\par

Next consider $T_{F_4}$ and $T_{F_5}$.
By Lemma \ref{intestimates} (4) it follows that
\begin{equation}\label{Fjest4-5}
\nm {F_4}{L^{r,\infty}_1}<\infty \quad \mbox{and} \quad
\nm {F_5}{L^{r,\infty}_1}<\infty
\end{equation}
when
$$
\begin{cases}
-t_0-t_1-t_2 + d / r \leq 0,
\quad & t_0 < d / r
\\[1 ex]
t_1 + t_2 > 0, \quad & t_0 =  d / r
\\[1 ex]
t_1 + t_2 \geq 0, \quad & t_0 > d / r.
\end{cases}
$$
If $t_0 < d / r$ and  $-t_0-t_1-t_2 + d / r \leq 0$, then $ t_1 + t_2 > 0$.
Therefore \eqref{Fjest4-5} holds when
$$
0 \leq t_1 + t_2
$$
and
$$
0  \leq t_0 + t_1 + t_2 - d / r,
$$
with $ 0< t_1 + t_2 $ when $ t_0 =  d / r $.
Hence Proposition \ref{prop1} (3) gives
\begin{equation*}
\nm {T_ {F_4} (u_1, u_2)}{L^{p_0'}} \lesssim \nm {f_1}{L ^ {p_1} _{t_1}}
\nm {f_2}{L ^ {p_2} _{t_2}}
\end{equation*}
for $r \geq \min (2, p_2)$, and (3) follows.

\par

Finally, by Proposition \ref{prop1} (2)
we get that
\begin{equation*}
\nm {T_ {F_5} (u_1, u_2)}{L^{p'_0}} \lesssim \nm {f_1}{L ^ {p_1} _{t_1}}
\nm {f_2}{L ^ {p_2} _{t_2}}
\end{equation*}
when
$r \geq \min (2, p_1)$. This gives (4), and the proof is complete.
\end{proof}

\par

In the following lemma we give another view to Lemma \ref{estimatesTFj},
which will be used for the proof of Theorem \ref{proptogether}.

\par

\begin{lemma} \label{estimatesT_F}
Let $F$, $F_j$ and $u_j$ be the same as in Lemma \ref{estimatesTFj}.
Furthermore, assume that \eqref{lastineq2A}, $ 0 \leq \masfR (p) \leq
1/2$, and \eqref{lastineq2B}
hold, with strict inequality in \eqref{lastineq2B}
when $ t_1, t_2 $ or $t_0$ is equal to $ d \cdot \masfR (p)$. Then
\begin{equation*}
\nm {T_ {F_j} (u_1, u_2)}{L^{q'_0}}
\lesssim  \nm {f_1}{L ^ {p_1} _{t_1}}\nm {f_2}{L ^ {p_2} _{t_2}}
\end{equation*}
holds for every $j \in \{1,\dots, 5 \}$.

\par

Furthermore, if the conditions in \eqref{lastineq2A} and \eqref{lastineq2B}
are violated, then at least one of the relations in {\rm{(1)-(5)}}
in Lemma \ref{estimatesTFj} is violated.
\end{lemma}

\par

\subsection{Proof of main results}\label{subsec4.3}

%
Next we prove Theorems \ref{proptogether} and \ref{proptogetherMod}.

\par

\begin{proof}[Proof of Theorem \ref{proptogether}]
First we note that $ 0\leq \mathsf R (p) \leq 1/2$ is not fulfilled when all $p_j\ge 2$
and at least one
of them is strictly larger than $2$. The similar fact is true if the condition
$ 0\leq \mathsf R (p) \leq 1/2$
is replaced by
\begin{equation}\label{ineqH}
0\le \masfR (p) \le \mathsf H (p),
\end{equation}
where $\mathsf H(p)=H_1(1/p_1,1/p_2,1/p_3)$ and $H_1$ is the same
as in Lemma \ref{1/2}. Hence, we may replace the condition
$ 0\leq \mathsf R (p) \leq 1/2$ by \eqref{ineqH} when proving the proposition.

\par

First we assume that
\begin{equation}\label{ineqR1}
\masfR (p)\le \frac 1{p_0}\quad \text{and}\quad
\masfR (p)\le \max \left (  \frac 12, \min \left ( \frac 1{p_1},\frac 1{p_2}  \right ) \right ),
\end{equation}
and that \eqref{lastineq2B}
holds and  $f_j \in L^{p_j} _{t_j}$, $ j = 1,2$.
We express $ f_1* f_2$ in terms of $T_F $ given by \eqref{TFmap} and
$F$ given by \eqref{Fdef} as follows. Let $\Omega_j$,  $j=1,\dots ,5$,
be the same as in \eqref{theomegasets}
after $\Omega _2$ has been modified into
$$
\Omega _2 = \sets{ (\xi, \eta) \in \rr {2d}}{\eabs {\xi -\eta }<
\delta \eabs \xi }\setminus \Omega _1.
$$
Then $\cup \Omega _j=\rr {2d}$, $\Omega _j\cap
\Omega _k$ has Lebesgue measure zero when  $j\neq k$, and
\begin{multline*}
(f_1*f_2) (\xi) \eabs \xi ^{-t_0}  = \int F (\xi, \eta ) u_1 (\xi - \eta) u_2 (\eta) d\eta
=T_F(u_1,u_2)
\\[1ex]
=T_{F_1}(u_1,u_2)+\cdots + T_{F_5}(u_1,u_2)
\end{multline*}
where $ u_j (\cdo ) = \eabs \cdo  ^{t_j} f_j $, $ j=1,2$,
and $ F_j = \chi_{\Omega _j}F$, $j = 1,\dots, 5.$

\par

Now, Lemma \ref{estimatesT_F} implies that the $L^{p_0'}$ norm of each
of the terms $ T_{F_j}$, $ j = 1,\dots, 5$
is bounded by $C \nm {f_1}{L ^ {p_1} _{t_1}}\nm {f_2}{L ^ {p_2} _{t_2}} $
for some positive constant $C$
which is independent of  $f_1 \in L^{p_1}_{t_1}(\rr d)$
and $f_2\in L^{p_2}_{t_2}(\rr d)$.

\par

Hence, $f_1 *f_2 \in L^{p_0'} _{-t_0}$ when \eqref{ineqR1} holds.
By duality, the same conclusion holds when  the roles for $p_j$, $j=0,1,2$ have
been interchanged. By straight forward computations it follows that \eqref{ineqH}
is fulfilled if and only if \eqref{ineqR1} or one of the dual cases of \eqref{ineqR1}
are fulfilled. This gives the result.

The assertion (2) follows from (1) and the relation $\mathscr
F(f_1*f_2)= (2\pi )^{d/2}\widehat f_1\cdot \widehat f_2$.
\end{proof}

\par

\begin{proof}[Proof of Corollary \ref{coro-local}]
Let  $f_j\in \FL _{s_j, loc} ^{q_j} (X)$, $ j = 1,2 $ and let
$\phi \in C_0^\infty (X)$. Then we choose $\phi _1=\phi $ and $\phi _2 \in
C_0^\infty (X)$ such that $\phi _2 =1$ on $\supp \phi$. Since $\phi _jf_j
\in \FL _{s_j}^{q_j}$, the right-hand side of
$$
f_1f_2 \phi  = (f_1\phi _1)(f_2 \phi _2)
$$
is well-defined, and defines an element in $ \FL _{-s_0}^{q'_0}$, in the view of Theorem \ref{proptogether} (2).
The corollary now follows from \eqref{FLqFLqlocemb} .
\end{proof}

\par

\begin{proof}[Proof of Theorem \ref{proptogetherMod}]
The assertion (1) follows immediately from (2), Fourier's inversion formula and the
fact that the Fourier transform maps $M^{p,q}_{s,t}$ into
$W^{q,p}_{t,s}$. Hence it suffices to prove (2).

We first consider the case when $p_j,q_j<\infty$ for $j=1,2$. Then
$\mathscr S$ is dense in $M^{p_j,q_j}_{s_j,t_j}$ for $j=1,2$. Since $M^{p,q}_{s,t}$
decreases with $t$, and the map
$f\mapsto \eabs \cdo ^{t_0}f$ is a bijection from $M^{p,q}_{s,t+t_0}$ to
$M^{p,q}_{s,t}$, for every choices of $p,q\in [1,\infty]$ and $s,t,t_0\in
\mathbf R$, it follows that we may assume that $t_j=0$, $j=0,1,2$.

\par

We have
\begin{multline}\label{STFTproducts}
V_{\phi}(f_1 f_2) (x,\xi )
=(2\pi ) ^{-d/2}
\big ( V_{\phi _1}f_1 (x,\cdo ) * V_{\phi _2}f_2 (x,\cdo )\big )(\xi ),
\\[1ex]
\phi =\phi _1 \phi _2,\ \phi _j,f_j\in \mathscr S(\rr d),\ j=1,2,
\end{multline}
which follows by straight-forward application of Fourier's inversion formula.
Here the convolutions between the factors $(V_{\phi _j}f_j)(x,\xi )$, where $j=1,2$
should be taken over the $\xi$ variable only.

\par

By applying the $L^{p'_0}$ norm with respect to the $x$ variables
and using H{\"o}lder's inequality we get
$$
\nm {V_{\phi}(f_1f_2) (\cdo ,\xi )}{L^{p'_0}} \le (2\pi ) ^{-d/2} (v_1*v_2)(\xi),
$$
where $v_j =  \nm {V_{\phi _j}f_j (\cdo ,\eta )}{L^{p_j}} $. Hence by applying
the $L^{q'_0}_{-s_0}$ norm on the latter inequality and using Theorem
\ref{proptogether} we get
$$
\nm {f_1f_2}{M^{p'_0,q'_0}_{-s_0,0}}\lesssim \nm {v_1}{L^{q_1}_{s_1}}
\nm {v_2}{L^{q_2}_{s_2}}
\asymp \nm {f_1}{M^{p_1,q_1}_{s_1,0}}\nm {f_2}{M^{p_2,q_2}_{s_2,0}},
$$
and (2) follows in this case, since $\mathscr S$ is dense in
$M^{p_j,q_j}_{s_j,0}$ for $j=1,2$.

\par

For general $p_j$ and $q_j$, (2) follows from the latter inequality
and Hahn-Banach's theorem.

\par

Finally, by interchanging the order of integration it  follows that
the same is true when $M^{p_j,q_j}_{s_j,t_j}$
is replaced by $W^{p_j,q_j}_{s_j,t_j}$, $j=0,1,2$. The proof is complete.
\end{proof}

\par

\par

In order to prove Proposition \ref{otpmimality1}, we recall some facts concerning
compactly supported distributions and modulation spaces. By Proposition
4.1 and Remark 4.6 in \cite{RSTT} we have
$$
M^{p,q}_{s,t}\cap \mathscr E' = W^{p,q}_{s,t}\cap \mathscr E' =
\mathscr FL^q_s \cap \mathscr E' ,
$$
and that for every compact set $K\subseteq \rr d$, then
$$
\nm f{M^{p,q}_{s,t}}\asymp \nm f{W^{p,q}_{s,t}}\asymp
\nm f{\mathscr FL^{q}_{s}},\quad f\in \mathscr S'(\rr d),\ \supp f\subseteq K.
$$
In particular,
$$
\nm {f\, \fy}{M^{p,q}_{s,t}}\asymp \nm {f\, \fy}{W^{p,q}_{s,t}}\asymp
\nm {f\, \fy}{\mathscr FL^{q}_{s}},\quad f\in \mathscr S'(\rr d),\ \fy \in C_0^\infty (K).
$$
By applying the Fourier transform, and using the fact that $\mathscr
FM^{p,q}_{s,t}=W^{q,p}_{t,s}$, we get
\begin{multline}\label{Fsuppidentities}
\nm {f* \fy}{M^{p,q}_{s,t}}\asymp \nm {f* \fy}{W^{p,q}_{s,t}}\asymp
\nm {f* \fy}{\mathscr FL^{q}_{s}},
\\[1ex]
f\in \mathscr S'(\rr d),\ \fy \in \mathscr FC_0^\infty (K).
\end{multline}


\par

\begin{proof}[Proof of Proposition \ref{otpmimality1}]
We only prove the result in the case when $p_0'<\infty$. The modifications to the case
when $p_0'=\infty$ are left for the reader.

\par

First we assume that (1) holds, and prove that \eqref{lastineq2A} must hold.
By duality it suffices to consider the
case $j=1$ and $k=2$. Let $f_0\in C_0^\infty (B_2(0))$ be such that
$0\le f_0\le 1$, $f_0 (x)=1$ when $x\in B_1(0)$. Here $B_r(a)$ is the open ball
centered at $x=a$ and with radius $r$. Also let $x_0\in \rr d$,
$f_1(x)=f_0(x-x_0)$ and $f_2(x)=f_0(x+x_0)$. Then it follows by straight-forward
computations that
$$
f_1*f_2 =f_0*f_0
$$
is independent of $x_0$, and that
$$
\nm {f_j}{L^{p_j}_{t_j}}\asymp \eabs {x_0}^{t_j}.
$$
In particular, $\nm {f_1*f_2}{L^{p_0'}_{-t_0}}>0$ is independent of $x_0$.

\par

Now if $(f_1,f_2)\mapsto f_1*f_2$ is continuous from $L^{p_1}_{t_1}\times
L^{p_2}_{t_2}$ to $L^{p_0'}_{-t_0}$, the inequality $\nm {f_1*f_2}{L^{p_0'}_{-t_0}}
\lesssim \nm {f_1}{L^{p_1}_{t_1}}\nm {f_2}{L^{p_2}_{t_2}}$ in combination
with the previous estimates imply that
$$
C\le \eabs {x_0}^{t_1+t_2},
$$
for some constant $C>0$ which is independent of $x_0\in \rr d$. By letting
$|x_0|$ tend to infinity, it follows from the latter relation that $t_1+t_2\ge 0$.
This proves that \eqref{lastineq2A} holds. If instead (2) or (3) hold, then
the same arguments show that \eqref{lastineq2A} still must hold.

\par

It remains to prove that \eqref{lastineq2B} must be true. Again we first consider
the case when (1) is true. By the first part
of the proof it follows that at least two of $t_0,t_1,t_2$ are non-negative,
and we may assume that $t_1\ge 0$, by duality. Let $\alpha \in (0,1]$, and
let $f_j(x)=\eabs x^{-t_j}e^{-\alpha |x|^2}$, $j=1,2$. Then
\begin{equation}\label{fjnorms}
\nm {f_j}{L_{t_j}^{p_j}}\asymp \alpha ^{-d/(2p_j)}.
\end{equation}

\par

Furthermore, if
$$
\Omega _x=\sets {y\in \rr d}{\eabs x/4\le |y|\le \eabs x/2},
$$
then $|x-y|\le 3\eabs x/2$, giving that
$$
-\alpha |y|^2\ge -\frac \alpha 4 -\frac \alpha 4\cdot |x|^2 \quad \text{and}
\quad -\alpha |x-y|^2\ge -\frac {9\alpha} 4 -\frac {9\alpha} 4\cdot |x|^2.
$$
Since $t_1\ge 0$ and $0<\alpha \le 1$ we obtain
\begin{multline}\label{gestimate}
g(x)\equiv (f_1*f_2)(x)\gtrsim \int _{\Omega _x}f_1(x-y)f_2(y)\, dy
\\[1ex]
\gtrsim \eabs x^{-t_1-t_2}\int _{\Omega _x}e^{-9\alpha |x|^2/4 -\alpha
|x|^2/4}\, dy
\\[1ex]
\gtrsim \eabs x^{-t_1-t_2} e^{-3\alpha |x|^2}\int _{\Omega _x} \, dy
\asymp
\eabs x^{d-t_1-t_2} e^{-3\alpha |x|^2}.
\end{multline}
Hence, if $h(x)=\eabs x^{d-t_1-t_2} e^{-3\alpha |x|^2}$, then
$\nm g{L_{-t_0}^{p_0'}}\gtrsim \nm h{L_{-t_0}^{p_0'}}$.
We need to estimate the right-hand side from below. Let $t=t_0+t_1+t_2$. Then
the result follows if we prove that $t\ge d\cdot \masfR (p)$.

\par

We have
\begin{multline}\label{hnormests}
\nm h{L_{-t_0}^{p_0'}}^{p_0'} \asymp \int (1+|x|)^{p_0'(d-t)}
e^{-3\alpha p_0'|x|^2}\, dx
\\[1ex]
\asymp \int _0^\infty r^{d-1}(1+r)^{p_0'(d-t)}e^{-3\alpha p_0'r^2}\, dr
\\[1ex]
\asymp \alpha ^{-d/2}\int _0^\infty r^{d-1}\left (1+\frac r{\alpha ^{1/2}}\right )
^{p_0'(d-t)}e^{-r^2}\, dr
\\[1ex]
\gtrsim \alpha ^{-d/2}\int _1^\infty r^{d-1}\left (1+\frac r{\alpha ^{1/2}}\right )
^{p_0'(d-t)}e^{-r^2}\, dr.
\end{multline}
Since $0<\alpha \le 1$ we get $1+r/\alpha ^{1/2}\asymp r/\alpha ^{1/2}$ when
$r\ge 1$. Hence \eqref{hnormests} gives
\begin{multline*}
\nm h{L_{-t_0}^{p_0'}}^{p_0'} \gtrsim
\alpha ^{-d/2}\int _1^\infty r^{d-1}\left (\frac r{\alpha ^{1/2}}\right )
^{p_0'(d-t)}e^{-r^2}\, dr
\\[1ex]
= \alpha ^{-(d(p_0'+1)-p_0't)/2}\int _1^\infty r^{d-1+p_0'(d-t)}e^{-r^2}\, dr
\\[1ex]
\asymp \alpha ^{-(d(p_0'+1)-p_0't)/2}.
\end{multline*}
That is
\begin{equation}\label{hnormests2}
\nm {f_1*f_2}{L_{-t_0}^{p_0'}}\gtrsim \alpha ^{-(d(1+1/p_0')-t)/2}
=
\alpha ^{-(d(2-1/p_0)-t)/2}.
\end{equation}

\par

By the assumptions we have
$$
\nm {g_1*g_2}{L^{p_0'}_{-t_0}}\lesssim \nm {g_1}{L_{t_1}^{p_1}}
\nm {g_2}{L_{t_2}^{p_2}},
$$
for every $g_1,g_2\in \mathscr S$. Hence \eqref{fjnorms} and
\eqref{hnormests2} give
\begin{equation}\label{alpha-q-relation}
\alpha ^{-(d(2-1/p_0')-t)/2}\lesssim \alpha ^{-d(1/p_1+1/p_2)/2},
\end{equation}
when $0<\alpha \le 1$. If we let $\alpha$ tend to zero, then it follows from
\eqref{alpha-q-relation} that
\begin{equation}\label{q-relation2}
d\cdot \left ( 2-\frac 1{p_0}\right ) -t\le d\cdot \left (\frac 1{p_1}+\frac 1{p_2}
\right ).
\end{equation}
Since $t=t_0+t_1+t_2$, the last relation is the same as \eqref{lastineq2B},
and the assertion follows.

\par

It remains to prove \eqref{lastineq2B} when (2) or (3) hold. We choose
$K=\overline{B_1(0)}$ and an element $0\le \fy \in \mathscr FC_0^\infty (K)$
such that $\fy (x)>0$ when $x\in K$. In order to find such element $\fy$, we
first let $0\le \psi \in C_0^\infty (B_{1/2}(0))$ be rotation invariant and such that
$\psi (0)>0$. Then $\fy =\mathscr F(\psi *\psi )$ satisfies the desired properties.

\par

Assume that $(f_1,f_2)\mapsto f_1*f_2$ is continuous from $M^{p_1,q_1}_{s_1,t_1}
\times M^{p_2,q_2}_{s_2,t_2}$ to $M^{p_0',q_0'}_{-s_0,-t_0}$, or from
$W^{p_1,q_1}_{s_1,t_1} \times W^{p_2,q_2}_{s_2,t_2}$ to
$W^{p_0',q_0'}_{-s_0,-t_0}$. If $\fy _0=\fy *\fy \ge 0$, then it follows from
\eqref{Fsuppidentities} that
\begin{equation}\label{continuityrel2}
\nm {f_1*f_2*\fy _0}{L^{p_0'}_{-t_0}}\lesssim
\nm {f_1*\fy }{L^{p_1}_{t_1}} \nm {f_2*\fy }{L^{p_2}_{t_2}}.
\end{equation}

\par

Now let $f_j(x)=\eabs x^{-t_j}e^{-\alpha |x|^2}$ as before. Then \eqref{fjnorms}
and Young's inequality gives
\begin{equation}\tag*{(\ref{fjnorms})$'$}
\nm {f_j*\fy}{L^{p_j}_{t_j}} \lesssim \alpha ^{-d/(2p_j)}.
\end{equation}
Furthermore, since $\psi \ge 0$ is non-zero in $\overline{B_1(0)}$, it follows from
\eqref{gestimate} that
$$
f_1*f_2*\fy _0 \gtrsim h*\psi _0,
$$
where $h=\eabs \cdo ^{d-t_1-t_2}e^{-3\alpha |\cdo |^2}$ are
the same as before, and $\psi _0=\fy _0$ in $B_1(0)$ and $\psi _0=0$ otherwise.

\par

Since $\eabs {x-y}^{d-t_t-t_2}\asymp \eabs x^{d-t_1-t_2}$ and
$e^{-3\alpha |x-y |^2}\gtrsim e^{-6\alpha |x |^2}$ when $|y|\le 1$, we get
$$
(f_1*f_2*\fy _0 )(x)\gtrsim \eabs x^{d-t_1-t_2} e^{-6\alpha |x |^2}.
$$
By the same arguments as in the proof of \eqref{hnormests2} we now obtain
\begin{equation}\tag*{(\ref{hnormests2})$'$}
\nm {f_1*f_2*\fy _0}{L_{-t_0}^{p_0'}}\gtrsim \alpha ^{-(d(2-1/p_0)-t)/2}.
\end{equation}

\par

A combination of \eqref{continuityrel2}, \eqref{fjnorms}$'$ and \eqref{hnormests2}$'$
now gives \eqref{alpha-q-relation} which in turn lead to \eqref{q-relation2} or equivalently
to \eqref{lastineq2B}. The proof is complete.
%
%
\end{proof}


\par

\end{document}